\documentclass[reqno]{amsart}
\usepackage{amsmath}
\usepackage{amssymb}
\usepackage{amsthm}
\usepackage{bbm}
\usepackage{graphicx}
\usepackage{float}
\usepackage[usenames]{color}
\usepackage{cite}

\usepackage[colorlinks,linkcolor=blue,anchorcolor=red,citecolor=blue]{hyperref}
\usepackage{a4wide}
\usepackage{marginnote}

\newcommand{\R}{\mathbb{R}}
\renewcommand{\d}{\mathrm{d}}
\newcommand{\E}{\mathcal{E}}

\newcommand{\one}{\mathbbm{1}}

\numberwithin{equation}{section}

\newtheorem{thm}{Theorem}[section]
\newtheorem{lemma}[thm]{Lemma}
\newtheorem{prop}[thm]{Proposition}

\newtheorem{coro}[thm]{Corollary}

\makeatletter
\@namedef{subjclassname@2020}{%
	\textup{2020} Mathematics Subject Classification}
\makeatother

\title[Moment propagation of an ionic Vlasov-Poisson system]{Moment propagation of a Vlasov-Poisson system for ion flow in the quasi-neutral regime}

\author[Z. Zhang]{Zhiwen Zhang}
\address[Z. Zhang]{Department of Mathematics, The Chinese University of Hong Kong,
	Shatin, Hong Kong SAR, P.R.~China}
\email{zwzhang@math.cuhk.edu.hk}
\date{\today}



\begin{document}
	\thispagestyle{empty}
	
	\begin{abstract}
		In light of recent work in the global well-posedness of solutions for an ionic Vlasov-Poisson system, as demonstrated by Griffin-Pickering and Iacobelli \cite{GSIV}, the current work focuses on the moment propagation of the corresponding system in quasi-neutral regime. Such moment propagation result relies on an estimate of $Q_*(t)=|V(t;0,x,v)-V(0;0,x,v)|$, where $V(s;t,x,v)$ represents the solution of the characteristic ordinary differential equation associated with the Vlasov-Poisson system. The main goal of this work is to serve the future research on quasi-neutral limit for ionic Vlasov-Poisson system in $\R^3$. 
	\end{abstract}
	
	\maketitle
	
	\section{Introduction}
	
	The Vlasov-Poisson equation is a fundamental equation in plasma physics and astrophysics that describes the behavior of a collision-less plasma in the presence of an electromagnetic field. In this article, we focus on studying the moment propagation of a specific variant known as the Vlasov-Poisson system for ion flow, alternatively referred to as the \emph{Vlasov-Poisson system with massless electrons (VPME)} \cite{GSIV}. The system is described by the following equations:
	
	\begin{equation}
		\label{vpme}
		(VPME)=\left\{
		\begin{aligned}
			& \partial_t f+v\cdot\nabla_x f+E\cdot \nabla_vf =0,\\
			& E=-\nabla U,\\
			& \varepsilon^2\Delta U=g(x) e^{U}-\rho=ge^U-\int_{\R^3}f(t,x,v)\d v,\\
			& f_{t=0}=f_{0}\geqslant 0.
		\end{aligned}
		\right.
	\end{equation}
	Here, the variable $t\geqslant 0$ represents time, while the phase-space is defined as $\R^3\times\R^3$, encompassing three-dimensional position coordinates $x$ and velocity components $v$. The distribution function of ions, denoted by $f=f(t,x,v)$, captures the behaviour of ions in the plasma, taking into account their positions and velocities.
	
	The function $g$, introduced and discussed in detail in \cite{GSIV}, is associated with the external potential experienced by the electrons in the system. This potential influences the behaviour of the ions through the electric field, which, in turn, affects the distribution function $f$. The function $g$ is assumed to satisfy certain mathematical properties, namely $g\in L^1\cap L^\infty(\R^3)$, ensuring its integrability and boundedness.
	
	The parameter $\varepsilon$ in the system equations corresponds to the Debye length of the plasma. The Debye length characterizes the distance scale over which charged particles interact with each other within the plasma. It plays a crucial role in determining the level of charge screening and the extent of deviations from quasi-neutrality in the plasma. The specific value of the Debye length depends on the particular characteristics and properties of the plasma under consideration.

	By studying the Vlasov-Poisson system with these considerations, this research aims to provide insights into the dynamics and behaviour of ions in a plasma, considering the influence of the Debye length and the external potential on the distribution function. The results obtained in this investigation contribute to a deeper understanding of plasma physics and its applications in various fields, including fusion energy research, astrophysics, and plasma processing technologies.
	
	The existence and uniqueness of solutions of Vlasov-Poisson system in electron case is first given by Schaeffer \cite{C2}, also by Horst \cite{C3}, and Lions and Perthame \cite{C1}. Global well-posedness of VPME was first proved by Han-Kwan and Iacobelli \cite{V1} in one-dimensional case. Later, Griffin-Pickering and Iacobelli \cite{GSIV} proved the well-posedness of VPME in three-dimensional space. The moment propagation of VPME in torus form is given by Griffin-Pickering and Iacobelli \cite{quasi}. This result helped the two authors solve the quasineutral limit problem of VPME in torus form.

	We begin by introducing the energy functional associated with the system (\ref{vpme})
	\begin{equation}
		\label{energy}
		\E[f]=\int_{\R^3}\int_{\R^3}|v|^2f \d x\d v+\varepsilon^2\int_{\R^3}|E|^2\d x+2\int_{\R^3}(U-1)ge^U\d x.
	\end{equation}
	Furthermore, in accordance with the well-posedness result presented in \cite{GSIV}, we establish the following proposition.
	
	\begin{prop}
		\label{wellposedness}
		Let $f_0 \in L^1 \cap L^{\infty}\left(\mathbb{R}^3 \times \mathbb{R}^3\right)$ be a probability density satisfying
		\begin{equation}
			\label{H1}
			\int_{\mathbb{R}^3 \times \mathbb{R}^3}|v|^{m_0} f_0(x, v) d x d v<+\infty \quad \text { for some } m_0>6
		\end{equation}
		
		\begin{equation}
			\label{H2}
			\quad f_0(x, v) \leq \frac{C}{(1+|v|)^r} \quad \text { for some } r>3 .
		\end{equation}
		Assume that $g \in L^1 \cap L^{\infty}\left(\mathbb{R}^3\right)$, with $g \geq 0$ satisfying $\int_{\mathbb{R}^3} g\d x=1$, and that $\mathcal{E}\left[f_0\right] \leq C$. Then there exists a unique global solution $f \in L^{\infty}\left([0, T] ; L^1 \cap L^{\infty}\left(\mathbb{R}^3 \times \mathbb{R}^3\right)\right)$ of (VPME) with initial datum $f_0$ such that $\rho_f \in L^{\infty}\left([0, T] ; L^{\infty}\left(\mathbb{R}^3\right)\right)$, for any finite $T>0$.
	\end{prop}
	
	Based on the established existence and uniqueness of the solution, we proceed to investigate the moment propagation problem in this paper.
	
	\subsection{Characteristic functions}
	
	Let us define the characteristic ordinary differential equations (ODEs) of the VPME system as follows
	
	\begin{align*}
		\frac{\d}{\d s}X\left(s;t,x,v\right) &= V\left(s;t,x,v\right), \quad X\left(t;t,x,v\right)=x;\\
		\frac{\d}{\d s}V\left(s;t,x,v\right) &= E\left(s;\left(X\left(s;t,x,v\right)\right)\right), \quad V\left(t;t,x,v\right)=v.
	\end{align*}
	
	It is evident that the distribution function can be expressed as
	\[f\left(t,x,v\right)=f\left(s,X\left(s;t,x,v\right),V\left(s;t,x,v\right)\right).\]
	Next, we introduce the quantity
	\begin{equation}
		\label{Qtd}
		Q\left(t,\delta\right)=\sup_{\left(x,v\right)\in\R^3_x\times\R^3_v} \int_{t-\delta}^t|E\left(s;0,x,v\right)|ds.
	\end{equation}
	We utilize $Q\left(t,t\right)$ to estimate the supremum of $\rho(t)$ and the velocity moments of the solution, which is defined by
	\begin{equation}
		\label{mkdef}
		M_k(t)=\sup_{s\in [0,t]}\int_{\R^3}\int_{\R^3}|v|^k f(s,x,v)\d x\d v.
	\end{equation}
	
	\subsection{Main Result}
	
	Now we introduce the main theorem
	
	\begin{thm}
		\label{mainthm}
		Let $f(t,x,v)\geqslant 0$ denotes a solution to the system \eqref{vpme} with initial data $f_0$. $f_0\in L^1\cap L^{\infty}(\R^3\times \R^3)$ satisfies
		\begin{equation}
			\label{Hypo}
			\int_{\R^3\times \R^3} |v|^{m_1} f_0(x, v) d x d v<+\infty,
		\end{equation} 
		where $m_1>2$ be a constant. Assume that $g \in L^1 \cap L^{\infty}\left(\mathbb{R}^3\right)$, with $g \geqslant 0$ satisfying $\int_{\mathbb{R}^3} g\d x=1$, and that $\mathcal{E}\left[f_0\right] \leqslant C$.
		
		Then for any positive $\omega<1$, it holds that the propagation of $Q(t,t)$ can be bounded by an exponential factor in terms of the Debye length and a polynomial factor in terms of the time parameter
		
		\begin{equation}
			\label{main}
			Q(t,t)\leqslant Ce^ {c\varepsilon^{-2}}\left[T^{\frac{1}{2}}+T^{1+\omega}\right],
		\end{equation}
		where $t\in[0,T]$, $C$ and $c$ only depend on $\E\left[f_0\right],\left\|f_0\right\|_{L^1},\left\|f_0\right\|_{L^\infty},M_{m_1}(0)$. 
		
		Furthermore, if $f_0$ satisfies \eqref{H2} additionally, the density of the system can be controlled by $Q(t,t)^3$, leading to the estimate,
		\begin{equation}
			\label{density}
			||\rho(t)||_{L^\infty(\R^3)}\leqslant C\left(1+Q(t,t)^3\right)
		\end{equation}
		for any $t>0$.
	\end{thm}
	
	\begin{coro}
		\label{rem}
		For any $k>2$ and $t\in [0,T]$, $M_k$ is defined in (\ref{mkdef}), then we have
		\begin{equation}
			\label{mk}
			M_k(t)\leqslant C\cdot 2^k\left(1+Q(t,t)^k\right),
		\end{equation}
		so under the assumptions of Theorem \ref{mainthm}, it holds that
		\begin{equation}
			\label{mkt}
			M_k(t)\leqslant C_k\cdot 2^k\left(1+e^{c\varepsilon^{-2}}t^{1+\omega}\right)^k.
		\end{equation}
	\end{coro}
	
	\begin{proof}
		For any $k>2$ and $s\in [0,t]$,
		\begin{align*}
			&\int_{\R^3}\int_{\R^3}|v|^kf(s,x,v)\d x\d v\\
			=&\int_{\R^3}\int_{\R^3}|V(s;0,y,w)|^kf(s,X(s;0,y,w),V(s;0,y,w))\d y\d w\\
			\leqslant &\int_{\R^3}\int_{\R^3}|w+Q(s,s)|^kf_0(y,w)\d y\d w\\
			\leqslant & \int_{\R^3}\int_{\R^3}2^k(|w|^k+|Q(s,s)|^k)f_0(y,w)\d y\d w\\
			=& 2^k\left[M_k(0)+M_0(0)Q(s,s)^k\right],
		\end{align*}
		which proves (\ref{mk}). Then we combine (\ref{main}) and (\ref{mk}), we deduce (\ref{mkt}).
	\end{proof}
	
	In previous work, Pallard \cite{pallard} established the inequality $Q(t,t)\leqslant C_1(T^{1/2}+T^{7/5})$ for the electronic Vlasov-Poisson equation in 2012. Notably, the time parameter in (\ref{main}) is similar to Pallard's result. Griffin-Pickering and Iacobelli \cite{GSIV} extended these findings to the case of $\varepsilon=1$, $k<m_0$, where $m_0$ satisfies (\ref{H1}). They proved the inequality
	\begin{equation*}
		\int_{\R^3}\int_{\R^3}|v|^kf(t,x,v)\d x\d v\leqslant \exp\left[C\left(1+\log\left(1+M_k(0)\right)\right)\exp\left(Ct\right)\right]
	\end{equation*}
	for any $t>0$, where $M_k(0)$ represents the left-hand side of (\ref{H1}).
	
	In this work, we estimate $Q(t,t)$ for an ionic Vlasov-Poisson system in the quasi-neutral regime. This result has implications for the quasi-neutral limit problem of the ionic Vlasov-Poisson system in the entire space. Griffin-Pickering and Iacobelli \cite{quasi} tackled the quasi-neutral limit problem in a torus and demonstrated that $Q(t,t)\leqslant C\varepsilon^{-2\frac{k-2}{k-3}}(t+1)^{\frac{k-2}{k-3}}$ holds for $k\in (3,\frac{13}{4}]$ in the context of the ionic Vlasov-Poisson system in a torus.
	
	To prove Theorem \ref{mainthm}, we first estimate the electronic field. In this regard, we employ a similar methodology as in \cite{GSIV} and \cite{quasi}. Secondly, we employ a four-step process to establish (\ref{main}). We adopt a classical approach by partitioning $[t-\delta,t]\times\mathbb{R}^3_x\times\mathbb{R}^3_v$ into three sets: the good set, the bad set, and the ugly set. Similar techniques have been employed in \cite{pallard,C1,C2,chen,glassey}.
	
	In this article, the constant $C$ is variable and depends solely on $f_0$ and $g$, unless otherwise stated.
	
	\section{Preliminaries}
	\subsection{The energy}
	
	The energy functional (\ref{energy}) assumes a crucial role in the analysis of Vlasov-Poisson equations. It captures essential aspects of the system's dynamics, stability, and behaviour. By exploring its properties and evolution, we gain valuable insights into the interplay between kinetic and potential energy, conservation laws, and emergent structures. These analytical methods have also found application in related works such as \cite{horst1,horst2,batt,glassey,C1}.

	\begin{lemma}
		\label{ener}
		Assume $\E[f_0]<+\infty$, then for any $t>0$,
		\begin{equation}
			\label{ener1}
			\E\left[f\right](t)= \E\left[f_0\right].
		\end{equation}
		Moreover, there exists a constant $C$ only depending  on $f_0$ and $g$ such that
		\begin{equation}
			\label{ener2}
			\int_{\R^3}\int_{\R^3}|v|^2f\d x\d v \leqslant C,
		\end{equation}
		and if $f_0\in L^\infty(\R^3\times \R^3)$, there exits a constant $C_{\frac{5}{3}}$ only depending on $f_0$ and $g$,such that
		\begin{equation}
			\label{E35}
			||\rho||_{L^{\frac{5}{3}}}\leqslant C_{\frac{5}{3}}.
		\end{equation}
		The energy function $\E[f]$ is defined in \eqref{energy}. 
	\end{lemma}
	
	\begin{proof}
		Firstly, times $|v|^2$ on both sides of the first equation in (\ref{vpme})
		\begin{equation}
			\label{E1}
			|v|^2\partial_tf+|v|^2v\cdot \nabla_xf+|v|^2E\cdot \nabla_vf=0.
		\end{equation}
		Take integration over (\ref{E1})
		\[\partial_t\int_{\R^3}\int_{\R^3}|v|^2f\d x\d v+\int_{\R^3}\int_{\R^3}|v|^2v\cdot \nabla_xf\d x\d v+\int_{\R^3}\int_{\R^3}|v|^2E\cdot\nabla_vf\d x\d v=0.\]
		Since $f\in L^1(\R^3\times\R^3)$, by partly integration
		\begin{equation}
			\label{newE2}
			\partial_t\int_{\R^3}\int_{\R^3}|v|^2f\d x\d v=2\int_{\R^3}\int_{\R^3} vf\cdot E\d x\d v=2\int_{\R^3}j\cdot E\d x\d v,
		\end{equation}
		where
		\[j=\int_{\R^3}vf(t,x,v)\d v.\]
		Note that
		\begin{equation}
			\label{E3}
			\begin{aligned}
				\rho_t+\nabla\cdot j&=\int_{\R^3}\partial_t f+v\cdot\nabla_x f\d v\\
				&=-\int_{\R^3}E\cdot\nabla_vf\d v\\
				&=0.
			\end{aligned}
		\end{equation}
		Now, we use (\ref{newE2}) and(\ref{E3})
		\begin{equation*}
			\begin{aligned}
				\frac{\partial}{\partial t}\int_{\R^3}|E|^2\d x=&2\int_{\R^3}E\cdot E_t\d x=2\int_{\R^3}\nabla U\cdot \nabla U_t\d x\\
				=&-2\int_{\R^3}U\Delta U_t\d x=-\frac{2}{\varepsilon^2}\int_{\R^3}U\cdot\frac{\partial}{\partial t}(ge^U-\rho)\d x\\
				=&-\frac{2}{\varepsilon^2}\frac{\partial}{\partial t}\int_{\R^3}(U-1)g e^U\d x+\frac{2}{\varepsilon^2}\int_{\R^3}U\cdot(-\nabla j)\d x\\
				=&-\frac{2}{\varepsilon^2}\frac{\partial}{\partial t}\int_{\R^3}(U-1)g e^U\d x+\frac{2}{\varepsilon^2}\int_{\R^3}\nabla U\cdot j\d x\\
				=&-\frac{2}{\varepsilon^2}\frac{\partial}{\partial t}\int_{\R^3}(U-1)g e^U\d x-\frac{2}{\varepsilon^2}\int_{\R^3}E\cdot j\d x\\
				=&-\frac{2}{\varepsilon^2}\frac{\partial}{\partial t}\int_{\R^3}(U-1)g e^U\d x-\frac{1}{\varepsilon^2}\frac{\partial}{\partial t}\int_{\R^3}|v|^2f\d x.
			\end{aligned}
		\end{equation*}
		The last equality uses (\ref{newE2}). Hence,
		\[\frac{\partial}{\partial t}\left[\int_{\R^3}\int_{\R^3}|v|^2f \d x\d v+\varepsilon^2\int_{\R^3}|E|^2\d x+2\int_{\R^3}(U-1)ge^U\d x\right]=0,\]
		which means
		\[\frac{\partial}{\partial t}\E\left[f\right]=0.\]
		Hence, we have (\ref{ener1})
		\[\E\left[f\right](t)=\E\left[f_0\right]=C.\]
		Note that for any $U\in\R$, we have $(U-1)e^U\geqslant -1$. So we have
		\begin{equation*}
			\begin{aligned}
				\int_{\R^3}\int_{\R^3}|v|^2 f\d x\d v &\leqslant C-2\int_{\R^3}(U-1)ge^U\d x\\
				&\leqslant C+2\int_{\R^3}g\d x.
			\end{aligned}
		\end{equation*}
		This inequality is \eqref{ener2}. Then by a classical analysis \cite[Chapter 4]{glassey}, \eqref{E35} holds. 
	\end{proof}
	
	\subsection{Electronic Field Estimate}
	
	Let's consider the third equation of (\ref{vpme})
	\begin{equation*}
		\varepsilon^2\Delta U=ge^U-\rho.
	\end{equation*}
	Decompose $U$ into $U=\bar{U}+\hat{U}$, where $\bar{U}$ satisfies
	\begin{equation*}
		-\varepsilon^2\Delta \bar{U}=\rho, \ \ \lim_{|x|\to\infty}\bar{U}(x)=0.
	\end{equation*}
	The remainder $\hat{U}$ satisfies
	\begin{equation*}
		\varepsilon^2\Delta \hat{U} =g e^{\bar{U}+\hat{U}}.
	\end{equation*}
	Correspondingly, $\bar{E}=-\nabla \bar{U}$ and $\hat{E}=-\nabla \hat{U}$.
	
	Now we recall the well-known Hardy-Littlewood-Sobolev inequality \cite[Theorem 4.5.3]{operator},
	\begin{prop}
		\label{convolution}
		Let $k_a(y)=|y|^{-\frac{n}{a}}$, $y\in \R^n$, if $1<a<\infty$ and $1\leqslant p<q\leqslant\infty$ and
		\[\frac{1}{p}+\frac{1}{a}=1+\frac{1}{q},\]
		then
		\[||k_a* u||_{L^q}\leqslant C_{p,a} ||u||_{L^p}.\]
	\end{prop}
	
	By the classical theory of Poisson's equation \cite{evans}, let
	\[\Phi(x)=\frac{1}{4\pi |x|}.\]
	Clearly, $\Phi$ is the fundamental solution of Laplace's equation, so
	\begin{equation}
		\label{U1}
		\bar{U}=\frac{1}{\varepsilon^2}\Phi*\rho,
	\end{equation}
	and
	\begin{equation}
		\label{U2}
		\hat{U}=-\frac{1}{\varepsilon^2}\Phi*(ge^{\bar{U}+\hat{U}}).
	\end{equation}
	
	Then we deduce a lemma
	
	\begin{lemma}
		\label{geu}
		Suppose $\rho\in L^1\cap L^{\frac{5}{3}}$, we have that for all $p\in[1,\infty]$
		\[||ge^U(t)||_{L^p}<C_pe^{c_0\varepsilon^{-2}}.\]
		The constant $C_p$ is only with respect ro $f_0$, $g$ and $p$, and with no respect to $t$,
		\[c_0=C_{\frac{3}{2},3}\times C_{\frac{5}{3}}^{\frac{5}{6}}.\]
		Here, $C_{\frac{3}{2},3}$ is the constant in Proposition \ref{convolution}, $C_{\frac{5}{3}}$ is the constant in Lemma \ref{ener}.
	\end{lemma}
	
	\begin{proof}
		By (\ref{U1}) and Proposition \ref{convolution}, we have
		\begin{equation*}
			\begin{aligned}
				||\bar{U}||_{L^\infty} &\leqslant C_{\frac{3}{2},3}\varepsilon^{-2}||\rho||_{L^{\frac{3}{2}}}\\
				&\leqslant C_{\frac{3}{2},3}\varepsilon^{-2}||\rho||_{L^{\frac{5}{3}}}^{\frac{5}{6}}||\rho||_{L^1}^\frac{1}{6}\\
				&\leqslant c_0\varepsilon^{-2}.
			\end{aligned}
		\end{equation*}
		By (\ref{U2}), Since $\Phi>0$ and $ge^{\bar{U}+\hat{U}}>0$, it follows that $\hat{U}=-\frac{1}{\varepsilon^2}\Phi*(ge^{\bar{U}+\hat{U}})<0$. Then for all $p\in [1,\infty]$
		\begin{equation*}
			\begin{aligned}
				||ge^U||_{L^p} &=||ge^{\bar{U}+\hat{U}}||_{L^p}\\
				&\leqslant||ge^{\bar{U}}||_{L^p}\\
				&\leqslant e^{||\bar{U}||_{L^\infty}}||g||_{L^\infty}^{1-1/p}||g||^{1/p}_{L^1}\\
				&\leqslant C_p e^{c_0\varepsilon^{-2}}.
			\end{aligned}
		\end{equation*}
		
	\end{proof}
	
	\section{Proof of the main theorem}
	Firstly, we prove (\ref{density}). Define
	\begin{equation*}
		Q_*(t)=\sup_{(x,v)\in\R^3_x\times\R^3_v}|V(t;0,x,v)-v|.
	\end{equation*}
	With direct computation, we have $Q_*(t)\leqslant Q(t,t)$. Then we get a bound of the density.
	
	\begin{lemma}
		\label{rhot}
		Let $t\geqslant 0$. Assume that (\ref{H2}) hold and $Q_*(t)$ defined as above is finite. Then
		\begin{equation*}
			||\rho(t)||_{L^\infty(\R^3)}\leqslant C(1+Q_*(t)^3).
		\end{equation*}
	\end{lemma}
	
	\begin{proof}
		By (\ref{H2}), we have
		\begin{equation*}
			(1+|v|^r)f_0(x,v)\leqslant C,
		\end{equation*}
		where $r>3$. Then we have
		\begin{equation}
			\label{eqn311}
			f(t,x,v)=f(0,X(0;t,x,v),V(0;t,x,v))\leqslant \frac{C}{1+|V(0;t,x,v)|^r}.
		\end{equation}
		Let $y=X(0;t,x,v)$ and $w=V(0;t,x,v)$, then we have $x=X(t;0,y,w)$ and $v=V(t;0,y,w)$. Therefore,
		\begin{equation}
			\label{eqn35}
			\begin{aligned}
				|V(0;t,x,v)|&\geqslant (|v|-|V(0;t,x,v)-v|)_+\\
				&=(|v|-|w-V(t;0,y,w)|)_+\\
				&\geqslant (|v|-\sup_{(z,u)}|V(t;0,z,u)-u|)_+\\
				&=(|v|-Q_*(t))_+.
			\end{aligned}
		\end{equation}
		We then plug (\ref{eqn35}) into (\ref{eqn311})
		\begin{equation}
			\label{eqn312}
			f(t,x,v)\leqslant\frac{C}{1+(|v|-Q_*(t))_+^r}.
		\end{equation}
		Take integration over (\ref{eqn312})
		\begin{align}
			\rho(t,x)&\leqslant \int_{\R^3}\frac{C}{1+(|v|-Q_*(t))_+^r}\d v \notag\\
			&\leqslant C\int_{0}^{\infty}\frac{s^2}{1-(s-Q_*(t))^r}\d s\notag\\
			&= C(\int_0^{Q_*(t)}s^2\d s+\int_{Q_*(t)}^{\infty}\frac{s^2}{1+(s-Q_*(t))^r}\d s)\notag\\
			&=C(I_1+I_2)\label{eqn313}.
		\end{align}
		Here,
		\begin{equation*}
			I_1=\frac{1}{3}Q_*^3(t),
		\end{equation*}
		and
		\begin{equation*}
			\begin{aligned}
				I_2&=\int_0^\infty\frac{(s+Q_*(t))^2}{1+s^r} \d s\\
				&\leqslant\int_{0}^{\infty}\frac{2s^2+2Q_*^2(t)}{1+s^r}\d s\\
				&\leqslant C(1+Q_*^2(t)).
			\end{aligned}
		\end{equation*}
		since $r>3$. Plug $I_1$ and $I_2$ back to (\ref{eqn313}), 
		\begin{equation*}
			\rho(t,x)\leqslant C(1+Q_*(t)^3).
		\end{equation*}
	\end{proof}

	Now we estimate $Q(t,t)$. We use 4 steps.
	
	\subsection{Step 1}
	\begin{lemma}
		\label{E2bound}
		There exists a constant $C$ such that
		\begin{equation*}
			||\hat{E}||_{L^\infty}\leqslant C\varepsilon^{-2}e^{c_0\varepsilon^{-2}}.
		\end{equation*}
	\end{lemma}
	
	\begin{proof}
		By (\ref{U2}) and $\hat{E}=-\nabla\hat{U}$, we have
		\begin{equation*}
			\label{E2}
			\hat{E}=-\frac{1}{4\pi \varepsilon^2}\frac{x}{|x|^3}*ge^U.
		\end{equation*}
		We use Proposition \ref{convolution} and Lemma \ref{geu},
		\begin{equation*}
			\begin{aligned}
				||\hat{E}||_{L^{\infty}}&\leqslant \frac{1}{4\pi\varepsilon^2}||\frac{1}{|x|^2}*ge^U||_{L^{\infty}}\\
				&\leqslant C\varepsilon^{-2}||ge^U||_{L^3}\\
				&\leqslant C\varepsilon^{-2} e^{c_0\varepsilon^{-2}}.
			\end{aligned}
		\end{equation*}
	\end{proof}
	
	\begin{prop} \label{Prop: 33}
		For any $0\leqslant\delta\leqslant t\leqslant T$ we have, for any $\gamma\in (0,1)$
		\begin{equation*}
			Q(t,\delta)\leqslant\frac{1}{\gamma} C\left[\delta \varepsilon^{-2}e^{c_0\varepsilon^{-2}}+\delta\varepsilon^{-2}Q(t,\delta)^{4/3}+\delta^{\frac{1}{2}}\varepsilon^{-2}(1+M_{2+\gamma}(T))^{\frac{1}{2}}\right].
		\end{equation*}
		$M_k$ is defined in (\ref{mkdef}) and $Q(t,\delta)$ is defined in (\ref{Qtd}).
	\end{prop}
	
	\begin{proof}
		Fix $(x_*,v_*)\in \R^3\times\R^3$, let $X_*(s)=X(s;0,x_*,v_*)$, $V_*(s)=V(s;0,x_*,v_*)$, by Lemma \ref{E2bound},
		\begin{align}
			\int_{t-\delta}^t|E(s,X_*(s))|\d s &=\int_{t-\delta}^t|\hat{E}(s,X_*(s))|\d s+\int_{t-\delta}^t |\bar{E}(s,X_*(s))|\d s\notag\\
			&\leqslant C(\delta \varepsilon^{-2}e^{c_0\varepsilon^{-2}}+\varepsilon^{-2}\int_{t-\delta}^t\int_{\R^3}\frac{1}{|x-X_*(s)|^2}\rho(s,x)\d x\d s).\label{S1}
		\end{align}
		Let
		\begin{equation*}
			\begin{aligned}
				I_*(t,\delta)&=\int_{t-\delta}^t\int_{\R^3}\frac{1}{|x-X_*(s)|^2}\rho(s,x)\d x\d s\\
				&=\int_{t-\delta}^t\int_{\R^3}\int_{\R^3}\frac{1}{|x-X_*(s)|^2}f(s,x,v)\d v\d x\d s.
			\end{aligned}
		\end{equation*}
		We use the standard method to estimate $I_*$, which means to split $[t-\delta,t]\times\R^3\times\R^3$ into three subsets, namely good set $G$, bad set $B$  and ugly set $U$. We use the split given by C. Pallard \cite{pallard}.
		\[G=\{(s,x,v)\in[t-\delta,t]\times\R^3\times\R^3:|v|<P \text{ or } |v-V_*(s)|<P \},\]
		\begin{align*}
			B=\{(s,x,v)\in[t-\delta,t]\times\R^3\times\R^3:|v|\geqslant P \text{ and } |v-V_*(s)|\geqslant P \\
			\text{and } |x-X_*(s)|\leqslant \Lambda(s,v)\},
		\end{align*}
		\begin{align*}
			U=\{(s,x,v)\in[t-\delta,t]\times\R^3\times\R^3:|v|\geqslant P \text{ and } |v-V_*(s)|\geqslant P \\
			\text{and } |x-X_*(s)|> \Lambda(s,v)\}.
		\end{align*}
		Here, $P=1+2^{10}Q(t,\delta)$, $\Lambda(s,v)=L(1+|v|^{2+\gamma})^{-1}|v-V_*(s)|^{-1}$ and $L>0$ is to be determined. Obviously that this split is reasonable because $[t-\delta,t]\times\R^3\times\R^3=G\cup B \cup U$. Naturally, we define $I_*^G,I_*^B,I_*^U$ to be the integrations over $G,B,U$, respectively.
		
		\noindent {\bf Good set}:
		
		By Proposition \ref{convolution}, for any $u\in L^{\frac{5}{3}}\cap L^{\infty}(\R^3)$
		\begin{equation}
			\label{p1c}
			||u*|x|^{-2}||_{L^\infty}\leqslant C||u||_{L^3}\leqslant C||u||_{L^{\frac{5}{3}}}^{5/9}||u||_{L^\infty}^{4/9}.
		\end{equation}
		Let
		\[\rho_G(s,x)=\int_{\{v:|v|<P \text{ and } |v-V_*(s)|<P\}}f(s,x,v) \d v,\]
		then
		\begin{equation*}
			||\rho_G(s,x)||_{L^{\frac{5}{3}}}\leqslant ||\rho(s,x)||_{L^{\frac{5}{3}}}\leqslant C,
		\end{equation*}
		\begin{equation*}
			||\rho_G(s,x)||_{L^\infty} \leqslant C||f(s)||_{L^{\infty}}P^3\leqslant CP^3.
		\end{equation*}
		Now we use (\ref{p1c}) to estimate $I_*^G$
		\begin{equation}
			\label{IG}
			I_*^G=\int_{t-\delta}^t||\rho_G*|x|^{-2}||_{L^{\infty}} \d x\leqslant C\delta P^{4/3}.
		\end{equation}
		
		\noindent {\bf Bad set:}
		
		Firstly,
		\begin{equation*}
			\begin{aligned}
				&\int_{\{x:|x-X_*(s)|\leqslant \Lambda(s,v)\}} \frac{1}{|x-X_*(s)|^{2}}f(s,x,v) \d x\\
				\leqslant &||f||_{L^{\infty}}\int_{\{x:|x-X_*(s)|\leqslant \Lambda(s,v)\}}\frac{1}{|x-X_*(s)|^{2}}\d x\\
				\leqslant & C\Lambda(s,v).
			\end{aligned}
		\end{equation*}
		Next, we have
		\begin{align}
			I_*^B\leqslant & C\delta \int_{\{v:|v|\geqslant P \text{ and } |v-V_*(s)|\geqslant P\}}\Lambda(s,v)\d v\notag\\
			=& C\delta \int_{\{v:|v|\geqslant P \text{ and } |v-V_*(s)|\geqslant P\}}\frac{L}{(1+|v|^{2+\gamma})|v-V_*(s)|}\d v\notag\\
			= & C\delta L \int_{\{v:|v|\geqslant P \text{ and } |v-V_*(s)|\geqslant P,|v|\leqslant |v-V_*(s)|\}}\frac{1}{(1+|v|^{2+\gamma})|v-V_*(s)|}\d v\notag\\
			&+\int_{\{v:|v|\geqslant P \text{ and } |v-V_*(s)|\geqslant P,  |v|> |v-V_*(s)|\}}\frac{1}{(1+|v|^{2+\gamma})|v-V_*(s)|}\d v\notag\\
			\leqslant &C\delta L\int_{\{v:|v|\geqslant P \text{ and } |v-V_*(s)|\geqslant P,|v|\leqslant |v-V_*(s)|\}}\frac{1}{|v|^{3+\gamma}}\d v\notag\\
			&+\int_{\{v:|v|\geqslant P \text{ and } |v-V_*(s)|\geqslant P,  |v|> |v-V_*(s)|\}}\frac{1}{|v-V_*(s)|^{3+\gamma}}\d v\notag\\
			\leqslant & C\delta L\int_{|v|\geqslant P}\frac{1}{|v|^{3+\gamma}}\d v+\int_{|v-V_*(s)|\geqslant P}\frac{1}{|v-V_*(s)|^{3+\gamma}}\d v\notag\\
			\leqslant & \frac{1}{\gamma}C\delta LP^{-\gamma} \notag\\
			\leqslant & \frac{1}{\gamma} C\delta L.\label{IB}
		\end{align}
		The last inequality uses the fact that $P>1$.
		
		\noindent {\bf Ugly set:}
		
		\begin{equation*}
			\begin{aligned}
				I_*^U(t,\delta) =&\int_{t-\delta}^{t}\int_{\R^3}\int_{\R^3}\frac{f(s,x,v)\one_U(s,x,v)}{|x-X_*(s)|^2}\d v\d x \d s\\
				=&\int_{\R^3}\int_{\R^3}\int_{t-\delta}^{t}\frac{\one_U(s,X(s),V(s))}{|X(s)-X_*(s)|^2}\d s f(t,x,v) \d v\d x.
			\end{aligned}
		\end{equation*}
		Here, for simplicity, we use $X(s)$ and $V(s)$ to denote $X(s;t,x,v)$ and $V(s;t,x,v)$.
		
		If for any $s\in[t-\delta,t]$, we have $(s, X(s), V(s))\not \in U$, then $I_*^U=0$.
		
		If not, then there exists $s_1\in[t-\delta,t]$, such that $(s_1,X(s_1),V(s_1))\in U$. By the definition of $U$, $|V(s_1)|\geqslant P$ and $|V(s_1)-V_*(s_1)|>P$. Then we claim that for any $s\in [t-\delta,t]$,
		\begin{equation}
			\label{v1}
			\frac{1}{2}|v|\leqslant |V(s)|\leqslant 2|v|,
		\end{equation}
		\begin{equation}
			\label{v2}
			\frac{1}{2} |v-v_*|\leqslant |V(s)-V_*(s)|\leqslant |v-v_*|.
		\end{equation}
		Let's prove (\ref{v1}) first. $|V(s_1)|\leqslant |V(t)|+|V(t)-V(s_1)|\leqslant |v|+Q(t,\delta)\leqslant|v|+\frac{1}{2}|V(s_1)|$, so $|V(s_1)|\leqslant 2|v|$. For any $s\in[t-\delta,t]$, 
		\[|V(s)|\leqslant|v|+|v-V(s)|\leqslant |v|+\frac{1}{2}P\leqslant|v|+\frac{1}{2}|V(s_1)|\leqslant 2|v|,\]
		\[|V(s)|\geqslant|v|-|v-V(s)|\geqslant |v|-\frac{1}{4}P\geqslant|v|-\frac{1}{4}|V(s_1)|\geqslant \frac{1}{2}|v|.\]
		Secondly, we prove (\ref{v2}) similarly. $|V(s_1)-V_*(s_1)|\leqslant |v-v_*|+|V(s)-v|+|V_*(s)-v_*|\leqslant |v-v_*|+\frac{1}{2}P\leqslant |v-v_*|+\frac{1}{2}|V(s_1)-V_*(s_1)|$, so $|V(s_1)-V_*(s_1)|\leqslant 2|v-v_*|$. For any $s\in [t-\delta,t]$, 
		\begin{equation*}
			\begin{aligned}
				|V(s)-V_*(s)|&\leqslant |v-v_*|+|v-V(s)|+|v_*-V_*(s)|\leqslant|v-v_*|+\frac{1}{2}P\\
				&\leqslant |v-v_*|+\frac{1}{2}|V(s_1)-V_*(s_1)|\leqslant 2|v-v_*|,
			\end{aligned}
		\end{equation*}
		\begin{equation*}
			\begin{aligned}
				|V(s)-V_*(s)|&\geqslant |v-v_*|-|v-V(s)|-|v_*-V_*(s)|\geqslant|v-v_*|-\frac{1}{4}P\\
				&\geqslant |v-v_*|-\frac{1}{4}|V(s_1)-V_*(s_1)|\geqslant \frac{1}{2}|v-v_*|.
			\end{aligned}
		\end{equation*}
		After proving the claims (\ref{v1},\ref{v2}), we deduce the following lemma:
		\begin{lemma}
			\label{Ugly}
			For any $(x,v)\in\R^3\times\R^3$,
			\begin{equation}
				\label{Ugly0}
				\int_{t-\delta}^t\frac{\one_U(s,X(s),V(s))}{|X(s)-X_*(s)|^2}\d s\leqslant C\frac{1+|v|^{2+\gamma}}{L}.
			\end{equation}
		\end{lemma}
		\begin{proof}
			By (\ref{v1}) and (\ref{v2}), $\Lambda(s,V(s))=L(1+|V(s)|^{2+\gamma})^{-1}|V(s)-V_*(s)|^{-1}\geqslant 2^{-3-\gamma}L(1+|v|^{2+\gamma})^{-1}|v-v_*|^{-1}=2^{-3-\gamma}\Lambda(t,v)$. Hence,
			\begin{equation}
				\label{Ugly1}
				\frac{\one_U(s,X(s),V(s))}{|X(s)-X_*(s)|^2}\leqslant \frac{\one_{\R^3\backslash B(X_*(s),2^{-3-\gamma}\Lambda(t,v))}(X(s))}{|X(s)-X_*(s)|^2}\leqslant h(|Y(s)|),
			\end{equation}
			where $Y(s)=X(s)-X_*(s)$, $h(u)=\min\{|u|^{-2},4^{3+\gamma}\Lambda(t,v)^{-2}\}$. Since $h$ is a non-increasing, we only need to find a lower bound of $|Y(s)|$. For any $s_0\in[t-\delta,t]$, because $Y(s)+\int_{s_0}^s(s-u)Y''(u)\d u=Y(s_0)+(s-s_0)Y'(s_0)$, we have
			\begin{align}
				|Y(s)|\geqslant & |Y(s_0)+(s-s_0)Y'(s_0)|-|\int_{s_0}^{s}(s-u)Y''(u)\d u|\notag\\
				\geqslant & |Y(s_0)+(s-s_0)Y'(s_0)|-(s-s_0)\int_{s_0}^s|E(u,X(u))|+|E(u,X_*(u))|\d u\notag\\
				\geqslant & |Y(s_0)+(s-s_0)Y'(s_0)|-(s-s_0)\cdot 2Q(t,\delta).\label{Ys}
			\end{align}
			Set $s_0\in[t-\delta,t]$ such that $s=s_0$ minimize $|Y(s)|^2$, then $(s-s_0)Y(s_0)\cdot Y'(s_0)\geqslant 0$, so 
			\begin{equation}
				\label{h1}
				|Y(s_0)+(s-s_0)Y'(s_0)|^2\geqslant |Y'(s_0)|^2|s-s_0|^2.
			\end{equation} 
			Since
			\begin{equation}
				\label{h2}
				|Y'(s_0)|=|V(s)-V_*(s)|\geqslant \frac{1}{2}|v-v_*|,
			\end{equation}
			and
			\begin{equation}
				\label{h3}
				2Q(t,\delta)\leqslant 2^{-9} P\leqslant 2^{-9} |V(s)-V_*(s)|\leqslant 2^{-8}|v-v_*|.
			\end{equation}
			Take (\ref{h1}),(\ref{h2}),(\ref{h3}) back to (\ref{Ys}),
			\begin{equation*}
				\begin{aligned}
					|Y(s)|\geqslant &|Y(s_0)+(s-s_0)Y'(s_0)|-2(s-s_0)Q(t,\delta)\\
					\geqslant &|Y'(s_0)||s-s_0|-2|s-s_0|Q(t,\delta)\\
					\geqslant &(|Y'(s_0)|-2Q(t,\delta))|s-s_0|\\
					\geqslant &\frac{1}{3}|v-v_*||s-s_0|.
				\end{aligned}
			\end{equation*}
			Now we estimate the integration of $h(|Y(s)|)$.
			\begin{figure}[H]
				\centering
				\includegraphics[scale=0.2]{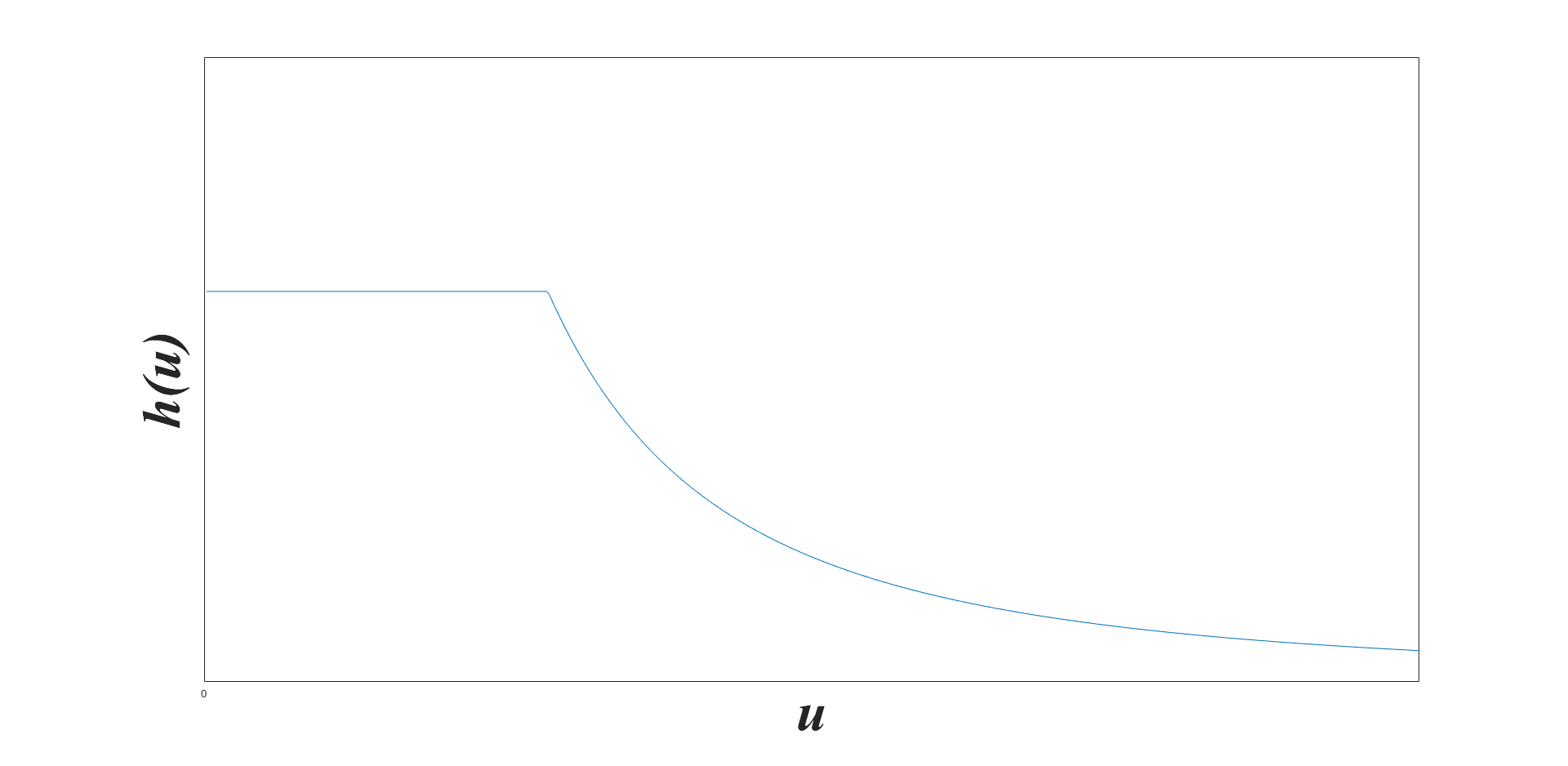}
				\caption{A brief graph of $h$}
			\end{figure}
			\begin{align}
				\int_{t-\delta}^t h(|Y(s)|)\d s\leqslant & \int_{t-\delta}^t h(\frac{|v-v_*||s-s_0|}{3})\d s\notag\\
				\leqslant &\int_{-\infty}^{\infty} h(\frac{|v-v_*||s-s_0|}{3})\d s\notag\\
				\leqslant &2\int_0^{\infty} h(\frac{|v-v_*|s}{3})\d s\notag\\
				\leqslant &\frac{6}{|v-v_*|}\int_0^{\infty}h(\eta)\d \eta\notag\\
				\leqslant &\frac{6}{|v-v_*|}\cdot C \cdot2^{3+\gamma}\Lambda (t,v)^{-1}\notag\\
				\leqslant &C\frac{1+|v|^{2+\gamma}}{L}.\label{Ugly2}
			\end{align}
			Combining equations (\ref{Ugly1}) and (\ref{Ugly2}), 
			\begin{align*}
				\int_{t-\delta}^{t} \frac{\one_U(s,X(s),V(s))}{|X(s)-X_*(s)|^2}\d s \leqslant& \int_{t-\delta}^t h(|Y(s)|)\d s\\
				\leqslant &C\frac{1+|v|^{2+\gamma}}{L},
			\end{align*}
			we arrive at (\ref{Ugly0}). Therefore, the proof of Lemma \ref{Ugly} is now concluded.
		\end{proof}
		By Lemma \ref{Ugly}, we have $I_*^U(t,\delta)\leqslant CL^{-1}(1+M_{2+\gamma}(T))$. Combining (\ref{S1}) and the estimation of $I_*^G,I_*^B,I_*^U$ (\ref{IG},\ref{IB}), we have got that
		\[Q(t,\delta)\leqslant\frac{1}{\gamma}C\left[\delta\varepsilon^{-2}e^{c_0\varepsilon^{-2}}+\delta\varepsilon^{-2}\delta L+\varepsilon^{-2}L^{-1}(1+M_{2+\gamma}(T))\right].\]
		Take $L=\delta^{-\frac{1}{2}}(1+M_{2+\gamma}(T))^{\frac{1}{2}}$,
		we get the theorem
		\[Q(t,\delta)\leqslant\frac{1}{\gamma} C\left[\delta \varepsilon^{-2}e^{c_0\varepsilon^{-2}}+\delta\varepsilon^{-2}Q(t,\delta)^{4/3}+\delta^{\frac{1}{2}}\varepsilon^{-2}(1+M_{2+\gamma}(T))^{\frac{1}{2}}\right].\]
		
	\end{proof}
	
	\subsection{Step 2}
	From this section, constants $C$ may depend on $\gamma$. $\gamma$ will be determined in section 3.4, so the constant in final result does not depend on $\gamma$.
	
	\begin{prop}
		For $t\in [0,T]$, $\exists c_*>0$, let
		\begin{equation*}
			\delta_*=c_*\varepsilon^{\frac{16}{7}}e^{-\frac{8}{7}c_0\varepsilon^{-2}}(1+M_{2+\gamma}(T))^{-\frac{1}{7}}.
		\end{equation*}
		Then for any $\delta\leqslant\delta_*$
		\begin{equation*}
			Q(t,\delta)\leqslant C\delta^{\frac{1}{2}}\varepsilon^{-2}e^{c_0\varepsilon^{-2}}\left[\delta^\frac{1}{2}+(1+M_{2+\gamma}(T))^{\frac{1}{2}}\right].
		\end{equation*}
	\end{prop}
	\begin{proof}
		For simplicity, let $H=1+M_{2+\gamma}(T)$. then by Proposition \ref{Prop: 33},
		\[Q(t,\delta)\leqslant C\varepsilon^{-2}e^{c_0\varepsilon^{-2}}\delta^{1/2}\left[\delta^{1/2}+\delta^{1/2}Q(t,\delta)^{4/3}+H^{1/2}\right].\]
		Since $\lim_{\delta\to 0} Q(t,\delta)=0$ and $Q(t,\delta)$ is non-decreasing for $\delta$, then for sufficiently small $\delta$, we have
		\begin{equation}
			\label{D1}
			\delta^{1/2}Q(t,\delta)^{4/3}\leqslant H^{1/2}.
		\end{equation} 
		If (\ref{D1}) holds, we have
		\begin{equation}
			\label{D2}
			Q(t,\delta)\leqslant C\varepsilon^{-2}e^{c_0\varepsilon^{-2}}\delta^{1/2}\left[\delta^{1/2}+H^{1/2}\right].
		\end{equation}
		So now, our goal is to find a $\delta_*$ such that for $\delta\leqslant \delta_*$, (\ref{D1}) holds.
		
		If for $\delta=t$, (\ref{D1}) still holds, the proof is completed.
		
		If not, $\exists \bar{\delta}$ such that
		\[\bar{\delta}^{1/2}Q(t,\bar{\delta})^{4/3}=H^{1/2}.\] 
		Then by (\ref{D2}), 
		\begin{align}
			H^{1/2}=& \bar{\delta}^{1/2}Q(t,\bar{\delta})^{4/3}\notag\\
			\leqslant &C_1\bar{\delta}^{1/2}\left(\varepsilon^{-2}e^{c_0\varepsilon^{-2}}\bar{\delta}^{1/2}\left[\bar{\delta}^{1/2}+H^{1/2}\right]\right)^{\frac{4}{3}}\notag\\
			=&C_1\bar{\delta}^{7/6}\varepsilon^{-\frac{8}{3}}e^{\frac{4}{3}c_0\varepsilon^{-2}}\left[\bar{\delta}^{1/2}+H^{1/2}\right]^{\frac{4}{3}}.\label{D3}
		\end{align}
		If we find $\delta_*$ such that $\delta_*\leqslant \bar{\delta}$, then for any $\delta\leqslant \delta_*$, (\ref{D1}) holds, and then (\ref{D2}) holds. If, for some $\delta_*$, there holds
		\begin{eqnarray}
			\label{D4}
			C_1\delta_*^{\frac{7}{6}}\varepsilon^{-\frac{8}{3}}e^{\frac{4}{3}c_0\varepsilon^{-2}}\left[\delta_*^{1/2}+H^{1/2}\right]^{\frac{4}{3}}\leqslant H^{1/2},
		\end{eqnarray}
		together with (\ref{D3}), we have $\delta_*\leqslant\bar{\delta}$.
		
		Let $C_2=(2^{\frac{4}{3}}C_1)^{-1}$, if the following two inequalities hold,  (\ref{D4}) holds.
		\begin{equation}
			\label{D5}
			\left\{
			\begin{aligned}
				& \delta_*^{\frac{7}{6}}\varepsilon^{-\frac{8}{3}}e^{\frac{4}{3}c_0\varepsilon^{-2}}\delta_*^{\frac{2}{3}}\leqslant C_2H^{1/2} ,\\
				& \delta_*^{\frac{7}{6}}\varepsilon^{-\frac{8}{3}}e^{\frac{4}{3}c_0\varepsilon^{-2}}H^{\frac{2}{3}}\leqslant C_2H^{1/2},
			\end{aligned}
			\right.
		\end{equation}
		which means
		\begin{equation}
			\label{D6}
			\left\{
			\begin{aligned}
				& \delta_*\leqslant C_2^{\frac{6}{11}}H^{\frac{3}{11}}\varepsilon^{\frac{16}{11}}e^{-\frac{8}{11}c_0\varepsilon^{-2}},\\
				& \delta_*\leqslant C_2^{\frac{6}{7}}H^{-\frac{1}{7}}\varepsilon^{\frac{16}{7}}e^{-\frac{8}{7}c_0\varepsilon^{-2}}.
			\end{aligned}
			\right.
		\end{equation}
		Since $\varepsilon\leqslant 1$ and $H\geqslant 1$, let $c_*=\min\{C_2^{\frac{6}{11}},C_2^{\frac{6}{7}}\}$, $\delta_*=c_*H^{-\frac{1}{7}}\varepsilon^{\frac{16}{7}}e^{-\frac{8}{7}c_0\varepsilon^{-2}}$. Now, we have that $\delta_*$ satisfies (\ref{D6}), and then (\ref{D5})an (\ref{D4}). Therefore, for any $\delta\leqslant\delta_*$, (\ref{D2}) holds.
	\end{proof}
	
	\subsection{Step 3}
	\begin{prop}
		\label{step3}
		For $t\in[0,T]$,
		\begin{equation}
			\label{s3}
			Q(t,t)\leqslant C\varepsilon^{-\frac{22}{7}}e^{\frac{11}{7}c_0\varepsilon^{-2}}(t^{\frac{1}{2}}+t)(1+M_{2+\gamma}(T))^{\frac{4}{7}}.
		\end{equation}
	\end{prop}
	
	\begin{proof}
		If $t\leqslant \delta_*$
		\[Q(t,t)\leqslant C\varepsilon^{-2}e^{c_0\varepsilon^{-2}}\left[t+t^{1/2}H^{1/2}\right].\]
		
		If $t>\delta_*$, let $n=\left\lfloor\frac{t}{\delta_*}\right\rfloor$, then we have $n\geqslant 1$ and $t-n\delta_*<\delta_*$. Split $[0,t]$ into $n+1$ small intervals,
		\[[0,t-n\delta_*],(t-(n-1)\delta_*,t-(n-2)\delta_*],(t-(n-2)\delta_*,t-(n-3)\delta_*]\dots(t-\delta_*,t].\]
		Then we have
		\begin{equation*}
			\begin{aligned}
				Q(t,t)\leqslant & Q(t-n\delta_*,t-n\delta_*)+\sum_{j=0}^{n-1}Q(t-j\delta_*,\delta_*)\\
				\leqslant & C(n+1)\delta_*^{1/2}\varepsilon^{-2}e^{c_0\varepsilon^{-2}}\left[\delta_*^{1/2}+H^{1/2}\right].
			\end{aligned}
		\end{equation*}
		Since $1\leqslant n\leqslant \frac{t}{\delta_*}$, we have $n+1\leqslant\frac{2t}{\delta_*}$. Hence,
		\begin{equation*}
			\begin{aligned}
				Q(t,t)\leqslant & C\frac{t}{\delta_*}\delta_*^{\frac{1}{2}}\varepsilon^{-2}e^{c_0\varepsilon^{-2}}\left[\delta_*^{1/2}+H^{1/2}\right]\\
				\leqslant & Ct\varepsilon^{-2}e^{c_0\varepsilon^{-2}}\left[1+\delta_*^{-1/2}H^{1/2}\right]\\
				\leqslant & Ct\varepsilon^{-2}e^{c_0\varepsilon^{-2}}\left[1+\varepsilon^{-8/7}e^{\frac{4}{7}c_0\varepsilon^{-2}}H^{\frac{4}{7}}\right]\\
				\leqslant & Ct\varepsilon^{-\frac{22}{7}}e^{\frac{11}{7}c_0\varepsilon^{-2}}(1+M_{2+\gamma}(T))^{\frac{4}{7}}.
			\end{aligned}
		\end{equation*}
		Put the two cases together, it yields (\ref{s3}).
	\end{proof}
	
	\subsection{Step 4}
	\begin{prop}
		\label{step4}
		For $t\in[0,T]$, and any $\omega<<1$,
		\begin{equation}
			\label{s4}
			Q(t,t)\leqslant Ce^{\frac{11}{5}c_0\varepsilon^{-2}}[T^{\frac{1}{2}}+T^{1+\omega}]\varepsilon^{-4}.
		\end{equation}
	\end{prop}
	
	\begin{proof}
		For $k>2+\gamma$,
			\begin{align}
				\int_{\R^3}\int_{\mathbb{R}^3}|v|^{2+\gamma}f(t,x,v)\d x \d v\leqslant & (\int_{\R^3}\int_{\mathbb{R}^3}|v|^{2}f(t,x,v)\d x\d v)^{\frac{k-2-\gamma}{k-2}} \notag\\
				& (\int_{\R^3}\int_{\mathbb{R}^3}|v|^{k}f(t,x,v)\d x\d v)^{\frac{\gamma}{k-2}}. \label{step4: e1}
			\end{align}
		Plug \eqref{mk} and \eqref{s3} into \eqref{step4: e1},
			\begin{align}
				M_{2+\gamma}(T)\leqslant& C_kM_k(T)^{\frac{\gamma}{k-2}} \notag\\
				\leqslant& C_k(M_k(0)+M_0(0)Q(t,t)^k)^{\frac{\gamma}{k-2}} \notag\\
				\leqslant & C_k (1+M_k(0))^{\frac{\gamma}{k-2}}(1+T)^{\frac{k\gamma}{k-2}}H^{\frac{4k\gamma}{7(k-2)}}\varepsilon^{-\frac{22\gamma}{7(k-2)}}. \label{step4: e2}
			\end{align}
		Since $M_k(0)+1,T+1,H,\varepsilon^{-1}\geqslant 1$ and $H=1+M_{2+\gamma}(T)$, using \eqref{step4: e2}, we have
		\begin{equation*}
			H\leqslant  C_k (1+M_k(0))^{\frac{\gamma}{k-2}} (1+T)^{\frac{k\gamma}{k-2}}H^{\frac{4k\gamma}{7(k-2)}}\varepsilon^{-\frac{22\gamma}{7(k-2)}}.
		\end{equation*}
		Take $k=m_1$, $m_1$ is the constant in \eqref{Hypo}. Clearly, we have $M_{m_1}(0)$ is a constant and $m_1>2$. Then for any $\omega>0$, we can find a $\gamma$ small enough, such that
		\begin{equation}
			\label{step4: e3}
			H^{\frac{4}{7}}\leqslant C(1+T)^{\frac{m_1}{m_1-2}\gamma}\varepsilon^{-\frac{22}{7(m_1-2)}\gamma}\leqslant C (1+T)^{\omega}\varepsilon^{-\frac{6}{7}}.
		\end{equation}
		Plug \eqref{step4: e3} into Proposition \ref{step3}, we obtain (\ref{s4}).
		
	\end{proof}
	
	Finally, we deduce the main theorem.
	\begin{proof}[Proof of Theorem \ref{mainthm}]
		According to Proposition \ref{step4}, we can deduce (\ref{main}).
		
		Applying Lemma \ref{rhot}, we obtain (\ref{density}).
		
		With that, we have successfully concluded the proof of Theorem \ref{mainthm}.
	\end{proof}
	
	{\bf Acknowledgements:} Zhiwen Zhang would like to thank his supervisor Prof. Renjun Duan for his patient guidance. The author also would like to thank anonymous referees for all valuable and helpful comments on the manuscript.


\begin{thebibliography}{99}
		\bibitem{GSIV}
		M. Griffin-Pickering, M. Iacobelli,
		Global strong solutions in $  {\mathbb{R}}^3 $ for ionic Vlasov-Poisson systems. {\it Kinetic and Related Models} {\bf 14} (2021) no.~4,  571--597. \href{https://www.aimsciences.org/article/doi/10.3934/krm.2021016}{doi: 10.3934/krm.2021016}
		
		
		\bibitem{operator}
		L. H\"{o}rmander,
		The analysis of linear partial differential operators I: Distribution theory and Fourier analysis.
		{\it Springer} (2015).
		\href{https://link.springer.com/book/10.1007/978-3-642-61497-2}{doi:10.1007/978-3-642-61497-2}
		
		
		\bibitem{quasi}
		M. Griffin-Pickering, M. Iacobelli,
		Stability in Quasineutral Plasmas with Thermalized Electrons.
		\href{https://arxiv.org/abs/2307.07561}{arXiv:2307.07561}
		
		
		
		\bibitem{pallard}
		C. Pallard,
		Moment propagation for weak solutions to the Vlasov–Poisson system. {\it Communications in Partial Differential Equations} {\bf 37} 2012, no.~7, 1273--1285.
		
		
		
		\bibitem{evans}
		L. C. Evans,
		Partial differential equations.
		{\it American Mathematical Society} (2022).
		
		
		
		\bibitem{C1}
		P.-L. Lions and B. Perthame,
		Propagation of moments and regularity for the 3-dimensional Vlasov-Poisson system.
		{\it Invent. Math.} {\bf 105} (1991), no.~2, 415--430.
		
		
		
		\bibitem{C2} 
		J. Schaeffer, 
		Global existence of smooth solutions to the Vlasov-Poisson system in three dimensions. {\it Comm. Partial Differential Equations} {\bf 16} (1991), no.~8-9, 1313--1335.
		
		
		
		\bibitem{C3} 
		E. R. Horst,
		On the asymptotic growth of the solutions of the Vlasov-Poisson system.
		{\it Mathematical methods in the applied sciences} {\bf 16}  (1993), no.~2, 75--85.
		
		
		
		
		\bibitem{V1} 
		D. Han-Kwan, M. Iacobelli,
		The quasineutral limit of the Vlasov-Poisson equation in Wasserstein metric. {\it Commun. Math. Sci.} {\bf 15} (2017), no.~2, 481--509.
		
		
		
		
		\bibitem{chen} 
		Z. Chen, J. Chen,
		Moments propagation for weak solutions of the Vlasov–Poisson system in the three-dimensional torus. {\it Journal of Mathematical Analysis and Applications} {\bf 472} (2019), no.~1, 728--737.
		
		
		
		\bibitem{glassey} 
		R. T. Glassey,
		The Cauchy problem in kinetic theory. Society for Industrial and Applied Mathematics.
		{\it Society for Industrial and Applied Mathematics} (1996).
		
		
		
		
		\bibitem{horst1} 
		E. R. Horst, H. Neunzert,
		On the classical solutions of the initial value problem for the unmodified non‐linear Vlasov equation I general theory.
		{\it Mathematical Methods in the Applied Sciences} {\bf 3} (1981), no.~1, 229--248.
		
		
		
		\bibitem{horst2} 
		E. R. Horst, H. Neunzert,
		On the classical solutions of the initial value problem for the unmodified non‐linear Vlasov equation II special cases.
		{\it Mathematical Methods in the Applied Sciences} {\bf 4} (1982), no.~1, 19--32.
		
		
		
		\bibitem{batt}
		J. Batt,
		Global symmetric solutions of the initial value problem of stellar dynamics.
		{\it Journal of Differential Equations} {\bf 25} (1977), no.~3, 342--364.
		
		
		
	\end{thebibliography}
\end{document}